\documentclass[11pt,reqno]{amsart}

\usepackage[T1]{fontenc}  
\usepackage[utf8]{inputenc}  
\usepackage[english]{babel} 

\usepackage[a4paper,hmargin=3cm,vmargin=3.5cm,heightrounded]{geometry}

\usepackage{xcolor}
\RequirePackage[colorlinks=True,linkcolor=blue,urlcolor=blue]{hyperref} 
\RequirePackage{graphicx} 
\usepackage{mathrsfs} 
\usepackage[font=scriptsize,labelfont={scriptsize}]{subfig} 
\usepackage{bbm} 
\usepackage{enumerate}
\usepackage{enumitem,kantlipsum}
\usepackage{caption} 
\usepackage{comment}
\usepackage[nocompress]{cite}
\usepackage{amssymb, amsmath}
\usepackage{nicefrac}
\usepackage{dsfont}

\numberwithin{equation}{section}
\setlist[enumerate]{leftmargin=*,label=(\roman*)}

\theoremstyle{plain}
\newtheorem{theorem}{Theorem}[section]
\newtheorem{proposition}[theorem]{Proposition}
\newtheorem{lemma}[theorem]{Lemma}
\newtheorem{corollary}[theorem]{Corollary}

\theoremstyle{remark}

\newtheorem{remark}[theorem]{Remark}

\newcommand{\eps}{\varepsilon}
\renewcommand{\phi}{\varphi}

\newcommand{\N}{\mathbb{N}}

\newcommand{\R}{\mathbb{R}}

\renewcommand{\P}{\mathbb{P}}
\newcommand{\E}{\mathbb{E}}

\newcommand{\cA}{\mathcal{A}}

\newcommand{\cF}{\mathcal{F}}

\newcommand{\cP}{\mathcal{P}}

\renewcommand{\d}{{\rm d}}

\newcommand{\cpl}{{\rm Cpl}}

\newcommand{\Lip}{{\rm Lip}}

\def\var{{\rm var}}

\def\Cb{{\rm UC}_{\rm b}}

\def\Bm{{\rm B}}
\def\Lipb{{\rm Lip}_{\rm b}}

\DeclareMathOperator{\OT}{OT}

\DeclareMathOperator*{\esssup}{ess\,sup}
\newcommand{\X}{E}

\begin{document}

\title[Hopf-Lax formula for controlled L\'evy processes]{Hopf-Lax approximation for value functions of L\'evy optimal control problems}

\author{Michael Kupper}
\address{Department of Mathematics and Statistics, University of Konstanz}
\email{kupper@uni-konstanz.de}

\author{Max Nendel}
\address{Department of Statistics and Actuarial Science, University of Waterloo}
\email{mnendel@uwaterloo.ca}

\author{Alessandro Sgarabottolo}
\address{Department of Mathematics, LMU München}
\email{sgarabottolo@math.lmu.de}

\date{\today}
\thanks{Financial support through the Deutsche Forschungsgemeinschaft (DFG, German Research Foundation) -- SFB 1283/2 2021 -- 317210226 is gratefully acknowledged. Part of this research was conducted during a visit of the second-named author at the Sydney Mathematical Research Institute}

\begin{abstract}
    In this paper, we investigate stochastic versions of the Hopf-Lax formula which are based on compositions of the Hopf-Lax operator with the transition kernel of a L\'evy process taking values in a separable Banach space.\ We show that, depending on the order of the composition, one obtains upper and lower bounds for the value function of a stochastic optimal control problem associated to the drift controlled L\'evy dynamics.\ Dynamic consistency is restored by iterating the resulting operators.\ Moreover, the value function of the control problem is approximated both from above and below as the number of iterations tends to infinity, and we provide explicit convergence rates and guarantees for the approximation procedure.

    \smallskip
		

\end{abstract}

\maketitle

\section{Introduction}

The study of basic deterministic optimal control problems in a separable Banach space $\X$ is closely related to the Hamilton-Jacobi equation
\begin{align*}
    \partial_t u & = H (\nabla_x u),\\
    u_0 & = f
\end{align*}
with a suitable initial condition $f\colon \X\to \R$ and a Hamiltonian $H\colon \X'\to \R$, where $\X'$ denotes the topological dual of $\X$. Classical PDE theory shows that, whenever $H$ is convex with superlinear growth and $f$ is bounded and Lipschitz continuous, it is possible to obtain a viscosity solution to the previous equation by using the celebrated Hopf-Lax formula
\[
u_t(x) = (\Phi_t f)(x) := \sup_{a\in \X} \Big( f(x + a) - t c\big( \tfrac{a}{t} \big) \Big),
\]
where $c\colon \X\to \R$ is such that $H$ is its Legendre transform, see, for example, \cite[\textsection10.3 Theorem 3]{evans2010pde} for the case $\X=\R^d$ with $d\in \N$ and a different sign convention.\ A key property that enables this result is the dynamic consistency of the Hopf-Lax operator, i.e., $\Phi_{s + t} f = \Phi_s \Phi_t f$ for all $s,t \ge 0$ and all bounded Lipschitz continuous functions $f\colon \X\to \R$.

In this paper, we study approximation schemes for value functions of optimal control problems with drift controlled L\'evy dynamics in separable Banach spaces based on the Hopf-Lax operator.\ 
To that end, we consider two families $(I_t)_{t\geq0}$ and $(J_t)_{t\geq0}$ of operators, defined by
\[
I_t f := \mu_t \Phi_t f \quad \text{and} \quad J_t f := \Phi_t \mu_t f\quad\text{for }f\in \Cb\text{ and }t\geq0,
\]
where $(\mu_t f)(x) := \int_\X f(x + y)\, \mu_t(\d y)$ for the family $(\mu_t)_{t \ge 0}$ of infinitely divisible distributions associated to the underlying L\'evy process and $\Cb$ denotes the space of all bounded uniformly continuous functions $\X\to \R$.

Clearly, the composition with the kernel $\mu_t$ does not preserve the dynamic consistency of the Hopf-Lax operator unless $\mu_t=\delta_0$ for all $t\geq0$.
However, we show that, for any convex function $c$ with superlinear growth, the dynamic consistency can be restored by considering iterative schemes of Lie-Trotter-type and the limit is given by the value function of the optimal drift control problem for the associated L\'evy dynamics with running cost $c$ and terminal cost $f\in \Cb$, i.e.,
\begin{align}\notag
\big(V_tf\big)(x):=\sup_{\alpha\in \cA} \E\bigg[f\bigg(x+Y_t+\int_0^t\alpha_u&\,\d u\bigg)-\int_0^t c(\alpha_u)\,\d u\bigg]\\
&=\lim_{n\to \infty} \big(I_{t/n}^nf\big)(x)=\lim_{n\to \infty}\big(J_{t/n}^n f\big)(x),
\label{eq.intro.main}
\end{align}
where $(Y_t)_{t\geq0}$ is an $\X$-valued L\'evy process on a filtered probability space $\big(\Omega,\cF,(\cF_t)_{t\geq0},\P\big)$ with $Y_t\sim\mu_t$ for all $t\geq0$, the set $\cA$ consists of all adapted $\X$-valued stochastic processes $(\alpha_t)_{t\geq0}$ with $\int_0^t\alpha_u\,\d u<\infty$ $\P$-a.s.\ for all $t\geq0$, and the limits are uniform in the state variable $x\in \X$.

We point out that nonlinear versions of Chernoff approximations as a generalization of classical 
Lie-Trotter theorems, which are usually given for linear semigroups, cf.\ \cite[Corollary 5.5]{pazy1983} and \cite{trotter1959}, have recently been derived in  \cite{BK22,BK23}.\
However, our proof of \eqref{eq.intro.main} does not invoke these results and is entirely self-contained.

Moreover, we show that the sequences of iterations $(I_{t/n}^n f)_{n \in \N}$ and $(J_{t/n}^n f)_{n \in \N}$ provide upper and lower bounds, respectively, for the value function $V_t f$ for all $t\geq 0$ and $f\in \Cb$.
As a byproduct of our main result, we also obtain the dynamic programming principle for the value function $(V_t)_{t\geq0}$ and recover the semigroup property for the family $(\Phi_t)_{t \ge 0}$.

The two stochastic Hopf-Lax approximations for $V_t f$ have been discussed in different contexts in the literature.\ The family of operators $(J_t)_{t \ge 0}$ produces the so-called Nisio semigroup, which is based on iterations of suprema over families of linear semigroups and corresponds to the choice of piecewise constant controls over refining partitions, see Corollary \ref{cor:dpp}.\ We refer to \cite{nendel2019upper,nisio1976} for a more general discussion of Nisio semigroups and their relation to optimal control problems.\
The family of operators $(I_t)_{t \ge 0}$, on the other hand, corresponds to worst-case expectations with an optimal transport penalties, cf.\ Lemma~\ref{lem:OT}.\ Therefore, the approximation from above in terms of $(I_{t/n}^n f)_{n \in \N}$ can be seen as a dynamically consistent worst-case expectation over nonparametric perturbations of the underlying L\'evy process.\ For the study of nonlinear semigroups generated by Wasserstein perturbations of transition semigroups, we refer to \cite{bartl2021randomwalks,nendel2023wasserstein}.\ In these works, the connection with the optimal control problem is established using viscosity theory, while we use an estimate that is based on structural arguments, cf.\ Lemma \ref{lem.key-estimate}.

For bounded Lipschitz continuous functions $f\colon E\to \R$, we provide an explicit uniform bound for the distance between $I_{t/n}^n f$ and $J_{t/n}^n f$ in terms of a constant depending only on the Lipschitz constant of $f$ and the cost function $c$ times $\frac{t}{n}$. In particular, both objects converge to the value function $V_t f$ of the control problem in \eqref{eq.intro.main} w.r.t.\ the supremum norm at a rate of $n^{-1}$.\ For bounded uniformly continuous functions $f\colon \X\to \R$, we provide explicit convergence guarantees in terms of the cost function $c$ and the inverse modulus of continuity of $f$, cf.\ Theorem \ref{thm.main}.

For similar approximations in a more general context, achieving a convergence rate of $n^{-1/4}$, we refer to \cite[Section~4.1]{blessing2023convergence}.\ In \cite{blessing2023convergence}, convergence rates for Chernoff-type approximations of strongly continuous convex monotone semigroups have been investigated.\
At the level of Hamilton–Jacobi–Bellman (HJB) equations, these correspond to the convergence rates of monotone schemes for viscosity solutions, see, e.g., \cite{BJ2005, BS1991, Krylov05, jiang2022, HLZ2020}. To establish the required consistency condition for HJB equations, one typically relies on convolution techniques, which usually require a finite-dimensional state space. 
In contrast to this approach, our method relies on direct arguments, using the specific structure of the dynamics, and provides linear convergence rates even for infinite-dimensional state spaces.

\section{Setup and preliminaries}
In this section, we state three auxiliary results which form the cornerstone of the proof of the main result, Theorem \ref{thm.main} below.\ Throughout, let $\X$ be a separable Banach space, $\Cb$ denote the space of all bounded and uniformly continuous functions $\X \to \R$, and $\cP$ be the set of all Borel probability measures on $\X$.\ The space $\Cb$ is endowed with the supremum norm $\|\cdot\|_\infty$ and the pointwise order $f \le g$ if $f(x) \le g(x)$ for all $x \in \X$ and $f,g\in \Cb$.\ An operator $\Gamma\colon \Cb\to \Cb$ is called a contraction if
\[
 \|\Gamma f-\Gamma g\|_\infty\leq \|f-g\|_\infty\quad \text{for all }f,g\in \Cb. 
\]
Moreover, let $\Lipb$ be the space of all $f\in \Cb$ with
$$
\|f\|_\Lip:=\sup_{x,y\in \X} \frac{|f(x)-f(y)|}{|x-y|}<\infty,
$$
where $|a|$ denotes the norm of $a\in \X$ and we use the convention $\frac{0}{0}=0$.
For $f\in \Cb$, we define
\[
\var f:= \sup_{x,y\in \X}|f(x)-f(y)|=\sup_{x\in \X} f(x)-\inf_{x\in \X}f(x)
\]
and, for all $\eps>0$,
\[
\omega_f^{-1}(\eps):=\sup\big\{\delta>0\colon |f(x)-f(y)|\leq \eps\text{ for all }x,y\in \X\text{ with }|x-y|\leq \delta\big\}.
\]
In the sequel, let $\Phi,\Psi\colon \Cb\to \Cb$ be two contractions and assume that
\begin{equation}\label{eq.Phi1}
\Phi f\geq f\quad \text{for all }f\in \Cb.
\end{equation}
We define
\[
If:= \Psi\Phi f\quad\text{and}\quad Jf:=\Phi\Psi f\quad\text{for all }f\in \Cb.
\]
Here and throughout, the product of two operators $\Cb\to \Cb$ is to be understood as their concatenation.

The following lemma provides the key estimate for our main result.
\begin{lemma} \label{lem.key-estimate}
For all $f\in \Cb$ and $n\in \N$,
\[
I^n f-J^nf \leq \sup_{x\in \X}\Big(\big(\Phi f\big)(x)-f(x)\Big).
\]
\end{lemma}
\begin{proof}
    Let $f\in \Cb$ and $n\in \N$.\ First, observe that, by definition, $$I^n f = \underbrace{\Psi\Phi \cdots \Psi \Phi}_{n\text{-times}}f=\Psi\underbrace{\Phi \Psi\cdots \Phi \Psi}_{(n-1)\text{-times}} \Phi f = \Psi J^{n-1} \Phi f.$$ Since $\Phi f\geq f$ and, both, $\Phi$ and $\Psi$ are contractions, we therefore obtain that
    \begin{align*}
        I^n f - J^n f &= \Psi J^{n-1} \Phi f - \Phi \Psi J^{n-1} f \leq  \Psi J^{n-1} \Phi f - \Psi J^{n-1} f\\
        &\le \|\Phi f - f\|_\infty= \sup_{x\in \X}\Big(\big(\Phi f\big)(x)-f(x)\Big).
    \end{align*}
    The proof is complete.
\end{proof}

Moving towards control problems, we now specialize to a particular form of the nonlinear operator $\Phi$.\ To that end, let $c \colon \X \to [0, \infty]$ be a convex function that grows superlinearly, i.e., $\lim_{|a| \to \infty} c(a)/|a|= \infty$, and satisfies $c(0)=0$. Then, for all $b\in [0,\infty)$,
\[
\overline c(b):=\sup_{a\in \X} \big(b|a|-c(a)\big)\in [0,\infty).
\]
For $f\in \Cb$ and $x\in \X$, we define
\begin{equation}\label{eq.def.Phi}
\big(\Phi f\big)(x):=\sup_{a\in \X} \big(f(x+a)-c(a)\big).
\end{equation}
Then, for all $f,g \in \Cb$ and $x,y\in \X$,
\[
\big|\big(\Phi f\big)(x) - \big(\Phi g\big)(y)\big| \le \sup_{a \in \X} |f(x + a) - g(y + a)|,
\]
which implies that $\Phi\colon \Cb\to \Cb$ is a contraction.\
Since $c(0)=0$, $\Phi$ also satisfies \eqref{eq.Phi1}.

\begin{proposition}\label{prop.guarantee}
  Let $\Phi$ be of the form \eqref{eq.def.Phi}.\ Then, for all $\eps>0$, $f\in \Cb$ and $n\in \N$,
\begin{equation}\label{eq.guarantee}
 I^nf -J^nf\leq \eps \quad\text{if}\quad \eps^{-1}\overline c\Bigg(\frac{\var f}{\omega_f^{-1}(\eps)}\Bigg)\leq 1.
\end{equation}
Moreover, for all $f\in \Lipb$ and $n\in \N$,
\begin{equation}\label{eq.lip.estimate}
 I^nf -J^nf\leq \overline c\big(\|f\|_\Lip\big).
\end{equation}
\end{proposition}

\begin{proof}
 Let $\eps>0$ and $f\in \Cb$ with
 \[
 \overline c\Bigg(\frac{\var f}{\omega_f^{-1}(\eps)}\Bigg)\leq \eps.
 \]
 Then,
 \[
  \sup_{x\in \X}\Big(\big(\Phi f\big)(x)-f(x)\Big)=\sup_{x\in \X} \sup_{a\in \X}\big(f(x+a)-f(x)-c(a)\big).
 \]
 Let $x\in \X$.\ If $a\in \X$ with $|a|< \omega_f^{-1}(\eps)$, then 
 \[
 f(x+a)-f(x)-c(a)\leq f(x+a)-f(x)\leq \eps
 \]
 and, if $a\in \X$ with $|a|\geq \omega_f^{-1}(\eps)$, then
 \begin{align*}
  f(x+a)-f(x)-c(a)&\leq \var f-c(a)\leq \frac{\var f}{\omega_f^{-1}(\eps)}|a|-c(a)\leq \overline c\Bigg(\frac{\var f}{\omega_f^{-1}(\eps)}\Bigg)\leq \eps.
 \end{align*}
 The claim in \eqref{eq.guarantee} thus follows from Lemma \ref{lem.key-estimate}.
 
 Now, let $f\in \Lipb$ and $x\in \X$. Then,
\[
\big(\Phi f\big)(x)-f(x)\leq \sup_{a\in \X} \big(\|f\|_\Lip |a|-c(a)\big)= \overline c\big(\|f\|_\Lip\big),
\]
which proves \eqref{eq.lip.estimate}.
\end{proof}

The previous proposition provides an upper bound for the difference of $I^n$ and $J^n$ in terms of the function $\overline c$.\ Later, we will rescale the cost function $c$ to a different cost $c_n$ for each $n\in \N$ with $\overline{c_n}=\frac1n \overline{c}$, leading to explicit convergence rates and guarantees for our Hopf-Lax approximation scheme.\ For this, it is essential to ensure that
the inequality $J f\leq If$ is valid for all $f\in \Cb$, leading to
\[
0\leq I^nf-J^nf\quad \text{for all }f\in \Cb\text{ and }n\in \N.
\]
This can be achieved for particular choices of $\Psi$.\ 
In view of this and our main result, Theorem \ref{thm.main} below, we restrict our attention to a special class of operators, which yield explicit descriptions of the operators $I$ and $J$, see the first line of \eqref{eq.JleqI} and Lemma \ref{lem:OT} below.\ In the sequel, let
\begin{equation}\label{eq.def.Psi}
\Psi f:= \inf_{\mu\in \cP} \Big(\mu f+\gamma(\mu)\Big)\quad\text{for all }f\in \Cb
\end{equation}
with $\gamma\colon \cP\to [0,\infty]$ satisfying $\inf_{\mu\in \cP}\gamma(\mu)=0$ and 
\begin{equation} \label{eq:def.mu}
(\mu f)(x) := \int_\X f(x+y)\,\mu(\d y)\quad \text{for all }\mu\in \cP,\, f\in \Cb\text{, and }x\in \X.
\end{equation}
Then, for all $f,g \in \Cb$ and $x,y \in \X$,
\begin{align*}
\big|\big(\Psi f\big)(x) - \big(\Psi g\big)(y)| & \le \sup_{\mu \in \cP} \bigg| \int_\X f(x + z)\,\mu(\d z) - \int_\X g(y + z)\,\mu(\d z) \bigg| \\
& \le \sup_{\mu \in \cP} \int_\X \bigg|f(x + z) - g(y + z) \bigg|\,\mu(\d z)\\
&\leq \sup_{a \in \X} |f(x+a)-g(y+a)|,
\end{align*}
which again implies that $\Psi\colon \Cb\to \Cb$ is a contraction.

Using the identity $\mu f-\infty=-\infty=\mu(-\infty)$ for $f\in \Cb$ and $\mu\in \cP$ as well as the convention $-\infty+\infty=-\infty$, we obtain the desired inequality
\begin{align}
\notag Jf=\Phi\Psi f&=\sup_{a\in \X} \inf_{\mu\in \cP} \Big(\mu\big(f(\cdot +a)-c(a)\big)+\gamma(\mu)\Big)\\
&\leq \inf_{\mu\in \cP}\Bigg(\mu \bigg(\sup_{a\in \X}\big(f(\cdot +a)-c(a)\big)\bigg)+\gamma(\mu)\Bigg)=\Psi\Phi f=If \label{eq.JleqI}
\end{align}
for all $f\in \Cb$.

We end this section with an explicit representation of the operator $I$ for $\Phi$ and $\Psi$ of the form \eqref{eq.def.Phi} and \eqref{eq.def.Psi}, respectively.
Recall that a coupling between two probability measures $\mu, \nu \in \cP$ is a probability measure $\pi$ on the Borel $\sigma$-algebra of the cartesian product $\X\times \X$ with first marginal $\mu$ and second marginal $\nu$.\ For all $\mu,\nu\in \cP$, we denote by $\cpl(\mu, \nu)$ the set of couplings between $\mu$ and $\nu$, and define the optimal transport problem associated with the cost function $c$ as
\[
\OT_c(\mu,\nu):=\inf_{\pi\in \cpl(\mu,\nu)} \int_{\X\times \X} c(z-y)\,\pi(\d y,\d z)\in [0,\infty].
\]

\begin{lemma} \label{lem:OT}
Let $\Phi$ and $\Psi$ be of the form \eqref{eq.def.Phi} and \eqref{eq.def.Psi}, respectively.\ Then, for all $f\in \Cb$,
\[
\big(If\big)(x)=\inf_{\mu\in \cP}\sup_{\nu\in \cP} \Big(\big(\nu f\big)(x)-\OT_c(\mu,\nu)+\gamma(\mu)\Big).
\]
\end{lemma}

\begin{proof}
Let $\Bm(\X)$ denote the space of all Borel measurable functions $\X\to \X$. Then, using
\cite[Theorem A.37]{foellmer2016finance} together with monotone convergence, we have that
\begin{align*}
\big(If\big)(x)&\leq \inf_{\mu\in \cP}\bigg(\int_{\X} \esssup_{a\in \Bm(\X)}\Big( f\big(x+y+a(y)\big)-c\big(a(y)\big)\Big)\,\mu(\d y)+\gamma(\mu)\bigg)\\
&=  \inf_{\mu\in \cP}\sup_{a\in \Bm(\X)}\bigg(\int_{\X}f\big(x+y+a(y)\big)-c\big(a(y)\big)\,\mu(\d y)+\gamma(\mu)\bigg)\\
&\leq \inf_{\nu\in \cP}\sup_{\nu\in \cP} \Big(\big(\nu f\big)(x)-\OT_c(\mu,\nu)+\gamma(\mu)\Big),
\end{align*}
where we have used that $\OT_c\big(\mu, \mu \circ [y \mapsto y + a(y)]^{-1} \big) \le \int_\X c\big(a(y)\big)\,\mu(\d y)$, choosing the comonotonic coupling $\mu\circ \big[y \mapsto \big(y,y + a(y)\big)\big]^{-1}$,   for all $\mu\in \cP$ and $a\in \Bm(\X)$.

Now, let $\mu,\nu\in \cP$ with $\inf_{\pi\in \cpl(\mu,\nu)} \int_{\X\times \X} c(z-y)\,\pi(\d y,\d z)<\infty$. Then,
\begin{align*}
\big(\nu f\big)(x)-\OT_c(\mu,\nu)+\gamma(\mu)&= \sup_{\pi\in \cpl(\mu,\nu)}\bigg( \int_{\X\times \X} f(x+z)-c(z-y)\,\pi(\d y,\d z)+\gamma(\mu)\bigg)\\
&\leq \int_\X \big(\Phi f)(x+y)\,\mu(\d y)+\gamma(\mu)=\big(\mu \Phi f\big)(x)+\gamma(\mu)
\end{align*}
for all $x\in \X$.\ Taking first the supremum over all $\nu\in \cP$ and then the infimum over all $\mu\in \cP$, the claim follows.
\end{proof}
We point out that a similar representation as in Lemma \ref{lem:OT} was also obtained in \cite{bartl2020computational},  for $\Psi=\mu$ with $\mu\in \cP$, using duality arguments and a version of Choquet's capacitability theorem. Moreover, in \cite{kupper2023risk}, the operator $I$, for $\Psi=\mu$ with $\mu\in \cP$, is defined as a risk measure with weak optimal transport penalty, and a partially more general version of the representation given in Lemma \ref{lem:OT}, is proved under additional constraints on the set of measures over which the supremum in $I$ is taken.\ This result is based on analytic selection arguments and an explicit representation of the extreme points of sets of measures defined in terms of generalized moment constraints, see \cite{weizsaecker1979irp,winkler1988extreme}.

\section{Main result}

As before, let $c \colon \X \to [0, \infty]$ be a convex function with $c(0)=0$ and superlinear growth, i.e., $\lim_{|a| \to \infty} c(a)/|a|= \infty$.\ For $t>0$ and $a\in \X$, we define $$c_t(a) := t c\big(a/t\big).$$ Then,
\begin{equation}\label{eq.overlinect}
\overline{c_t}(b)=\sup_{a\in \X} \big(b|a|-tc(a/t)\big)=t\sup_{a\in \X} \big(b|a/t|-c(a/t)\big)=t\overline c(b)\quad\text{for all }b\in[0,\infty).
\end{equation}
Throughout, let $(\mu_t)_{t>0}\subset \cP$ be an infinitely divisible family of probability measures and $(Y_t)_{t \ge 0}$ be an $\X$-valued L\'evy process on a filtered probability space $(\Omega, \cF, (\cF_t)_{t \ge 0}, \P)$ with $Y_t\sim \mu_t$ for all $t>0$.\ We point out that we do not impose any restrictions on the L\'evy process or its associated infinitely divisible distribution, not even completeness of the underlying probability space.\ In particular, {every} L\'evy process is covered by our setup. Prominent examples for L\'evy processes are Brownian Motions with drift and compound Poisson processes. We refer to \cite{applebaum2009levy,linde1986probability,sato2013levy} for a detailed discussion and examples of L\'evy processes.

For $t>0$ and $f\in \Cb$, let
\[
(\Phi_t f)(x) := \sup_{a \in \X} \big(f(x + a) - t c(a/t)\big)\quad\text{for all } x\in \X
\]
and $\mu_t f$ be given by \eqref{eq:def.mu} with $\mu = \mu_t$.\ We then consider
\[
I_tf:= \mu_t \Phi_tf\quad\text{and}\quad J_tf:=\Phi_t\mu_t f \quad \text{for all } t \ge 0 \text{ and } f \in \Cb.
\]
For notational convenience, we set $I_0f:=J_0f:=\Phi_0f:=f$ for all $f\in \Cb$.

Moreover, we define the set of admissible controls $\cA$ as the set of all adapted stochastic processes $\alpha\colon [0,\infty)\times \Omega \to \X$ with
\[
 \int_0^t |\alpha_u|\,\d u < \infty \quad\P\text{-a.s.\ for all }t\geq 0.
\]
For $t\geq 0$ and $x\in \X$, we set
\[
Y_t^x:=x+Y_t,
\]
and, for $\alpha\in \cA$, we consider the controlled dynamics
\[
Y_t^{x,\alpha} := x + Y_t+\int_0^t \alpha_u\,\d u, \quad \text{for } t \ge 0\text{ and } x \in \X.
\]
Then, the value function of the associated optimal control problem with running cost $c$ and terminal cost $f\in \Cb$ is given by
\begin{equation} \label{eq:def.Vt}
(V_t f)(x) := \sup_{\alpha \in \cA} \E \bigg[ f\big(Y_t^{x,\alpha}\big) - \int_0^t c(\alpha_u)\,\d u \bigg] \quad \text{for } t\geq0\text{ and }x\in \X.
\end{equation}
Observe that $V_t f\in \Cb$ for all $t\geq 0$ and $f\in \Cb$. Moreover, $V_0f=f$ for all $f\in \Cb$ and $$\| V_t f_1-V_t f_2\|_\infty\leq \|f_1-f_2\|_\infty \quad\text{for all }t\geq 0\text{ and }f_1,f_2\in \Cb.$$ 
The following two lemmata provide fundamental estimates for the interplay between the families $(I_t)_{t\geq 0}$, $(J_t)_{t\geq 0}$, and $(V_t)_{t\geq 0}$.

\begin{lemma}\label{lem.VI}
 For all $t\geq 0$ and $f\in \Cb$,
 \[
 V_tf\leq I_tf.
 \]
\end{lemma}

\begin{proof}
 Let $t\geq 0$ and $f\in \Cb$.\ For $t=0$, the statement is trivial.\ We therefore consider the case $t>0$, and define $\mu_t^\alpha:=\P\circ \big(Y_t+\int_0^t\alpha_u\,\d u\big)^{-1}$ for all $\alpha\in \cA$. Then, choosing the particular coupling $\P\circ \big(Y_t,Y_t+\int_0^t\alpha_u\,\d u\big)^{-1}\in \cpl(\mu_t,\mu_t^\alpha)$ and using Jensen's inequality, 
\[
\inf_{\pi \in \cpl(\mu_t,\mu_t^\alpha)} \int_{\X \times \X} c_t(z-y)\pi(\d y, \d z)\le \E\bigg[t c\bigg(\frac1t \int_0^t \alpha_u\,\d u \bigg)\bigg]\le \E\bigg[\int_0^t c(\alpha_u)\,\d u\bigg]
\]
for all $\alpha\in \cA$.\ Hence, by Lemma \ref{lem:OT},
\[
V_tf= \sup_{\alpha\in \cA}\bigg(\mu_t^\alpha f-\E\bigg[\int_0^t c(\alpha_u)\,\d u\bigg]\bigg)\leq \sup_{\alpha\in \cA}\Big(\mu_t^\alpha f-\OT_{c_t}(\mu_t,\mu_t^\alpha)\Big)\leq I_t f.
\]
The proof is complete.
\end{proof}

\begin{lemma}\label{lem.apriori}
 Let $s,t\geq 0$ and $f\in \Cb$. Then,
 \begin{equation}\label{eq.apriori}
  V_{s+t}f\leq V_sV_tf\quad\text{and}\quad V_sJ_t f\leq V_{s+t}f.
 \end{equation}
\end{lemma}

\begin{proof}
For $t=0$, both statements in \eqref{eq.apriori} are trivial.\ Therefore, assume that $t>0$.\ Let $x\in \X$ and $\alpha \in \cA$. Since $Y_{s+t}-Y_s$ is independent of $\cF_s$ with $Y_{s+t}-Y_s\sim \mu_t$, it follows that
\begin{align*}
     \E \bigg[ f\big(Y_{s+t}^{x,\alpha}\big) - \int_0^{s+t} c(\alpha_u)\,\d u \bigg]
    & = \E \bigg[ \E \bigg[ f\bigg(Y_s^{x,\alpha}+(Y_{s+t}-Y_s)+\int_{0}^{t}c(\alpha_{s+u})\, \d u\bigg) \\
    &\qquad \qquad\qquad\qquad -\int_0^{t} c(\alpha_{s+u})\,\d u \, \bigg|\, \cF_s \bigg] -\int_0^s c(\alpha_u)\,\d u\bigg] \\
    & \leq \E \bigg[(V_tf)\big(Y_s^{x,\alpha}\big) - \int_0^s c(\alpha_{u})\,\d u \bigg]\\
    &\leq (V_sV_tf)(x).
\end{align*}
Taking the supremum over all $\alpha\in \cA$, it follows that $\big(V_{s+t} f\big)(x)\leq \big(V_sV_tf\big)(x)$.

 Using again \cite[Theorem A.37]{foellmer2016finance} together with monotone convergence and the fact that $Y_{s+t}-Y_s$ is independent of $\cF_s$ with $Y_{s+t}-Y_s\sim \mu_t$, we find that
\begin{align*}
     \E \bigg[ \big(J_tf\big)&\big( Y_s^{x,\alpha}\big) - \int_0^{s} c(\alpha_u)\,\d u \bigg]\\
    & = \E \bigg[ \sup_{a\in \X}\bigg(\big(\mu_t f\big)\big(Y_s^{x,\alpha}+ta\big) -\int_{0}^{s}c(\alpha_{u})\, \d u-tc(a)\bigg)\bigg] \\
    &\leq \E \bigg[ \esssup_{\beta\in \cA}\bigg(\big(\mu_t f\big)\big(Y_s^{x,\alpha}+t\beta_s\big) -\int_{0}^{s}c(\alpha_{u})\, \d u-tc(\beta_s)\bigg)\bigg] \\
    &= \sup_{\beta\in \cA}\E \bigg[ \big(\mu_t f\big)\big(Y_s^{x,\alpha}+t\beta_s\big) -\int_{0}^{s}c(\alpha_{u})\, \d u-tc(\beta_s)\bigg] \\
     &= \sup_{\beta\in \cA} \E \bigg[ f\bigg(Y_{s+t}^x+\int_0^s\alpha_u\,\d u+t\beta_s\bigg) -\int_{0}^{s}c(\alpha_{u})\, \d u-tc(\beta_s)\bigg)\bigg]\\
    &\leq \sup_{\beta\in \cA} \E \bigg[ f\bigg(Y_{s+t}^x+\int_0^s\alpha_u\,\d u+\int_s^{s+t}\beta_u\,\d u\bigg) -\int_{0}^{s}c(\alpha_{u})\, \d u-\int_s^{s+t}c(\beta_u)\,\d u\bigg)\bigg]\\
    & \leq \big(V_{s+t}f\big)(x).
\end{align*}
Taking the supremum over all $\alpha\in \cA$, we may conclude that $\big(V_sJ_tf\big)(x)\leq \big(V_{s+t} f\big)(x)$.
\end{proof}

We are now ready to prove the main result on the convergence of the stochastic Hopf-Lax approximation.

\begin{theorem}\label{thm.main}
Let $t\geq 0$ and $f\in \Cb$.
\begin{enumerate}
\item[a)] For all $n\in \N$,
    \begin{equation}\label{eq.main1}
     J_{t/n}^nf\leq V_tf\leq I_{t/n}^n f
    \end{equation}
    and, for all $\eps >0$,
    \begin{equation}\label{eq.main1.1}
    \big\|I_{t/n}^nf-J_{t/n}^nf\big\|_\infty\leq \eps \quad\text{if}\quad n\geq \frac{t}\eps \overline c\Bigg(\frac{\var f}{\omega_f^{-1}(\eps)}\Bigg).
    \end{equation}
    Moreover,
    \begin{equation}\label{eq.main1.2}
     \sup_{n\in \N}J_{t/n}^nf=V_tf=\inf_{n\in \N}I_{t/n}^nf.
     \end{equation}
\item[b)] If $f\in \Lipb$, then
\begin{equation}\label{eq.main2}
    \big\|I_{t/n}^nf-J_{t/n}^nf\big\|_\infty\leq \tfrac{t}{n}\overline c\big(\|f\|_\Lip\big)\quad\text{for all }n\in \N.
    \end{equation}
\end{enumerate}
\end{theorem}

\begin{proof}
Choosing $s=0$ in Lemma \ref{lem.apriori}, it holds $J_tf\leq V_tf$ for all $t\geq 0$ and $f\in \Cb$. Now, assume that $J_{t/n}^nf\leq V_tf$ for all $t\geq 0$, $f\in \Cb$, and some $n\in \N$. Then, by Lemma \ref{lem.apriori}, for all $t\geq 0$ and $f\in \Cb$,
\[
 J^{n+1}_{t/(n+1)}f\leq V_{nt/(n+1)}J_{t/(n+1)}f\leq V_tf.
\]
We have therefore shown, by induction, that $J_{t/n}^nf\leq V_f$ for all $t\geq 0$, $f\in \Cb$, and $n\in \N$. Moreover, by Lemma \ref{lem.VI} and Lemma \ref{lem.apriori},
\[
V_tf\leq V_{t/n}^nf\leq I_{t/n}^nf
\]
for all $t\geq 0$, $f\in \Cb$, and $n\in \N$.

Using \eqref{eq.overlinect}, the statement in \eqref{eq.main1.1} and claim b) follow from \eqref{eq.guarantee} and \eqref{eq.lip.estimate}, respectively, and it remains to prove \eqref{eq.main1.2}. To that end, let $t\geq0$ and $f\in \Cb$.\ Then, by \eqref{eq.main1},
\[
\sup_{n\in \N} J_{t/n}^nf\leq V_t f\leq \inf_{n\in \N} I_{t/n}^nf,
\]
which, together with \eqref{eq.main1.1}, implies that
\[
\sup_{n\in \N}J_{t/n}^nf=\inf_{n\in \N} I_{t/n}^nf.
\]
The proof is complete.
\end{proof}

As a consequence of Lemma \ref{lem.apriori} and Theorem \ref{thm.main}, we obtain the classical Hopf-Lax formula for the family $(\Phi_t)_{t\geq 0}$, see, e.g., \cite[\textsection3.3 Lemma 1]{evans2010pde} for the $\R^d$-case with $d\in \N$ and different sign convention, and the dynamic programming principle for the family $(V_t)_{t\geq 0}$. 

\begin{corollary} \label{cor:dpp}
 For all $s,t\geq0$ and $f\in \Cb$,
 \[
 \Phi_{s+t}f=\Phi_s\Phi_t f, \quad J_{s+t}f\leq J_sJ_t f,\quad I_{s+t}f\geq I_sI_t f,\quad \text{and}\quad V_{s+t}f=V_sV_tf.
 \]
\end{corollary}

\begin{proof}
We start by proving the dynamic programming principle $V_{s+t}f=V_sV_t f$ for all $s,t\geq0$ and $f\in \Cb$.\ To that end, let $s,t\geq 0$, $f\in \Cb$, and $\eps>0$. Then, by Theorem \ref{thm.main}, there exists some $n\in \N$ such that
\[
I_{s/n}^nV_t\leq J_{s/n}^nV_t f +\tfrac\eps2\quad\text{and}\quad I_{s/n}^n\leq J_{s/n}^n f +\tfrac\eps2.
\]
Then, using Lemma \ref{lem.apriori} together with the fact that $J_{s/n}$ is a contraction, we find that
\begin{align*}
V_{s+t}f&\leq V_sV_tf\leq I_{s/n}^nV_t f\leq J_{s/n}^nV_t f+\tfrac\eps2\leq J_{s/n}^nI_{t/n}^n f+\tfrac\eps2\leq J_{s/n}^nJ_{t/n}^n f+\eps\leq V_{s+t}f+\eps.
\end{align*}
Since $\eps>0$ was arbitrary, it follows that $V_sV_tf=V_{t+s}f$ for all $s,t\geq0$ and $f \in \Cb$.

Choosing $Y_t=0$ or, equivalently, $\mu_t=\delta_0$, and $\cF_t=\{\emptyset,\Omega\}$ for all $t\geq 0$, it follows that $V_t=\Phi_t$ for all $t\geq 0$, so that $\Phi_{s+t}f=\Phi_s\Phi_t f$ for all $s,t\geq0$ and $f\in \Cb$.\ Choosing, respectively, $\gamma(\mu)=\infty$ for all $\mu\in \cP\setminus\{\mu_t\}$ and $\mu\in \cP\setminus\{\mu_s\}$, which leads to $\Psi=\mu_t$ and $\Psi=\mu_s$ in \eqref{eq.def.Psi}, we obtain from \eqref{eq.JleqI} that
\[
 J_tJ_s f=\Phi_t \mu_t\Phi_s\mu_sf\geq \Phi_t\Phi_s \mu_t\mu_s f=\Phi_{t+s}\mu_{t+s}f=J_{t+s}f
\]
and
\[
 I_tI_s f=\mu_t\Phi_t \mu_s\Phi_sf\leq \mu_t\mu_s\Phi_t\Phi_sf=\mu_{t+s}\Phi_{t+s}f=I_{t+s}f
\]
for all $s,t\geq0$ and $f\in \Cb$.
\end{proof}

\begin{remark}
 While all results of this section are concerned with suprema appearing in the operators $I_t$, $J_t$, and $V_t$ for $t\geq0$, they easily transfer to infima instead of suprema with $c$ being replaced by $-c$ in the definition of these operators.\ Considering $\overline I_t f:= -I_t(-f)$, $\overline J_t f:= -J_t(-f)$ and $\overline V_t f:=-V_t(-f)$ for $t\geq 0$ and $f\in \Cb$, one obtains the analogous results for the infimal operators with reversed inequality symbols, except in norm estimates, infima being replaced by suprema, and vice versa.
\end{remark}

\begin{remark}
    Theorem \ref{thm.main}, in particular, the explicit rate in \eqref{eq.main2}, suggests the possibility to develop numerical schemes to approximate the value function of the stochastic control problem \eqref{eq:def.Vt} based on an iteration of the operators $I$ and $J$. In \cite{kupper2023risk}, numerical methods for generalized versions of the operator $I$ have been developed. For example, in the case $c(a) = |a|^p$ for all $a\in \X$, the optimization inside the integral can be moved outside, becoming an optimization over sets of $p$-integrable functions, which leads to
    \[
    (If)(x) = \sup_{a \in L^p(\mu;\X)} \int_\X f(x + y + a(y)) - |a(y)|^p \,\mu(\d y),
    \]
    where $L^p(\mu;\X)$ is the space of $\mu$-equivalence classes of functions $a \in \Bm(\X)$ such that $\int_\X |a(y)|^p\,\mu(\d y) < \infty$.
    In a finite-dimensional setting, this optimization problem is then tackled using standard universal approximation results, see \cite{hornik1989universal,hornik1991approximation}.\ The idea is to perform the optimization over finite-dimensional sets of functions that are dense in $L^p$, e.g., neural networks.\ A similar approach could be employed for a functional approximation of the one-step operator $I_{t/n} f$.\ To include the dependence on the $x$ variable, one could discretize the space $\X$ and solve the optimization on a fixed grid, interpolating the resulting function in between.\ Another possibility that seems to be natural, given the continuity of the problem, is to interpolate the optimizing vector field rather than the value function.\
    One can also tackle the entire problem with an approximation over functions $a \colon \X \times \X \to \X$. If these functions are parametrized by neural networks, the optimization can be performed by sampling points $x$ from $\X$ during the training phase according to a chosen reference measure $\eta \in \cP$.\ A natural candidate could be, for example, a gaussian distribution concentrated on a region of the space where one is interested in computing the value function.

    With similar arguments, in the case $c(a)=|a|^p$ for $a\in \X$, one can choose a reference measure $\eta \in \cP$ with finite moment of order $p$ to sample the space variable $x$ and perform an optimization that returns an optimal $a \colon \X \to \X$ such that
    \[
    (Jf)(x) \approx \int_\X f(x + a(x) + y)\,\mu(\d y) - |a(x)|^p.
    \]
\end{remark}

\end{document}